\documentclass[11pt,reqno]{amsart} 

\usepackage{amssymb,amsmath,latexsym}
\usepackage{cite} 
\usepackage{graphicx}
\usepackage{epstopdf}
\usepackage{float}
\usepackage[height=190mm,width=130mm]{geometry} 

\theoremstyle{plain}
\newtheorem{theorem}{Theorem}

\theoremstyle{definition}

\theoremstyle{remark}

\numberwithin{equation}{section} 

\begin{document}
	\title{Domain-structured chaos for   discrete random processes}

\author{Akhmet Marat}
\address{Middle East Technical University\\ Department of Mathematics\\ Ankara\\ Turkey}
\email{marat@metu.edu.tr}

\begin{abstract}
We   introduce   the notion   of   domain-structured chaos  and apply  it  to   establish  a connection between   stochastic dynamics and  deterministic chaos.
\end{abstract}
\maketitle

In this paper, a new efficient method is used to show that a discrete random process exhibits chaotic dynamics. To achieve this goal, the concepts of abstract similarity map and domain-structured chaos, previously considered in our papers\cite{AkhmetSimilarity,AkhmetDomainStruct,AkhmetAbstFract,AkhmetCube, AkhmetBook}, are further developed. It must be noted that the concept of domain-structured chaos is more abstract in the present paper than in\cite{AkhmetDomainStruct},   since it  is  free of topological  assumptions.

	\section{Domain structured chaos} 
\label{sec1}

Let the metric space $(F,d)$   be given,   with  the  distance $d.$   Assume   that  there  exists  the  following presentation for  the set $F,$ 
\begin{equation} \label{AbstFracSet}
\mathcal{F} =  \big\{\mathcal{F}_{i_1 i_2 ... i_n ... } : i_k=1,2, ..., m, \; k=1, 2, ... \big\},
\end{equation}
where  $m$  is a natural number.  
The  presentation means  that  each   element  of the  set  $ F$  is   labeled  through at  least  by  one  member of the set  $\mathcal F,$   and each   element  of the set    $\mathcal F$  presents a  member from $F.$       We   assume that  the   uniqueness  is not  necessarily  required  for this relation.   That  is, if 
$\mathcal{F}_{i_1 i_2 ... i_n \ldots}$  and $\mathcal{F}_{j_1 j_2 ... jn, \ldots}$   present the same   element  $f \in \mathcal F,$   it  is not  necessary  that   $i_n=j_n$ for all $n =1,2,\ldots.$    Moreover,   we    keep  the distance $d$   for  the  set  $\mathcal F$  considering 
$d(\mathcal{F}_{i_1 i_2 ... i_n \ldots}, \mathcal{F}_{j_1 j_2 ... jn, \ldots} ) =  d(f_1,f_2),$  if  	$\mathcal{F}_{i_1 i_2 ... i_n \ldots}$  and $\mathcal{F}_{j_1 j_2 ... jn, \ldots}$   are presentations of   elements $f_1$ and $f_2$  of the set  $F$ respectively,   such  that  
$d(\mathcal{F}_{i_1 i_2 ... i_n \ldots}, \mathcal{F}_{j_1 j_2 ... jn, \ldots} ) = 0$  for   different  presentations of the same point in $F.$    
In what  follows, we   call  the  set  $\mathcal F$ a \textit{pre-chaotic structure}  for  the  set  $F.$			  It  is clear that  there  is a naturally  determined map  $\varPhi: F \to  F,$    which  does not  satisfy  the  uniqueness condition,  and values $\varPhi(f)$  for  a fixed $f$ constitute  the  set $\{\varphi(\mathcal{F}_{i_1 i_2 ... i_n \ldots})\},$  with   all labels of the point $f.$ 
In what  follows, we shall describe chaotic  properties of  the  map  $\varPhi$  in terms of the  dynamics of the map  $\varphi.$  

The following  sets are  needed, 
\begin{equation} \label{AbstFracSubSet}
\mathcal{F}_{i_1 i_2 ... i_n} = \bigcup_{j_k=1,2, ..., m } \mathcal{F}_{i_1 i_2 ... i_n j_1 j_2 ... },
\end{equation}
where indices  $ i_1, i_2, ..., i_n,$  are fixed.

It is clear that
\[ \mathcal{F} \supseteq \mathcal{F}_{i_1} \supseteq \mathcal{F}_{i_1 i_2} \supseteq ... \supseteq \mathcal{F}_{i_1 i_2 ... i_n} \supseteq \mathcal{F}_{i_1 i_2 ... i_n i_{n+1}} ... , \; i_k=1, 2, ... , m, \; k=1, 2, ... \, ,\]
that is, the sets form a nested sequence. 
Let us introduce the map $ \varphi : \mathcal{F} \to \mathcal{F} $ such that
\begin{equation} \label{MapDefn}
\varphi (\mathcal{F}_{i_1 i_2 ... i_n ... }) = \mathcal{F}_{i_2 i_3 ... i_n ... }.
\end{equation}
Considering iterations of the map, one can verify that
\begin{equation} \label{MapSubset}
\varphi^n(\mathcal F_{i_1 i_2 ... i_n}) = \mathcal{F},
\end{equation}
for arbitrary natural number $ n $ and $ i_k=1,2, ..., m, \; k=1, 2, ... \, $. The relations (\ref{MapDefn}) and (\ref{MapSubset}) give us a reason to call $ \varphi $ a \textit{similarity map} and the number $ n $ the \textit{order of similarity}.
The  similarity  map $\phi$  is known  as the Bernoulli shift \cite{Wiggins88}   for the symbolic dynamics.
We will  say   that for the sets $ \mathcal{F}_{i_1 i_2 ... i_n} $ the \textit{diameter condition} is valid,  if 
\begin{equation} \label{Diamprop}
\max_{i_k=1,2, ..., m} \mathrm{diam}(\mathcal{F}_{i_1 i_2 ... i_n}) \to 0 \;\; \text{as} \;\; n \to \infty,
\end{equation}
where $ \mathrm{diam}(A) = \sup \{ d(\textbf{x}, \textbf{y}) : \textbf{x}, \textbf{y} \in A \} $, for a set $ A $ in $ \mathcal{F} $.

Denote  by $ d(A, B)= \inf \{ d(\textbf{x}, \textbf{y}) : \textbf{x}\in A, \, \textbf{y} \in B \} $ the  function of    two bounded sets $ A $ and $ B $ in $ \mathcal{F} .$   Set $ \mathcal{F} $ satisfies the \textit{separation condition }of degree $ n $ if there exist a positive number $ \varepsilon_0 $ and a natural number $ n $  such that for arbitrary   indices  $ i_1 i_2 ... i_n $ one can find   indices $ j_1 j_2 ... j_n $ such  that
\begin{equation} \label{C2}
d \big( \mathcal{F}_{i_1 i_2 ... i_n} \, , \, \mathcal{F}_{j_1 j_2 ... j_n} \big) \geq \varepsilon_0.
\end{equation}

It   is clear   that  the  couple $(\mathcal F, d)$  is not, in general,  a metric space. Nevertheless,  it  is true  that  the  map  $\phi$  is continuous with  respect  to  the  distance $d,$  and all the   attributes of    definitions  for  Poincar\'{e},  Li-Yorke and Devaney  chaos can  be extended for  the  dynamics.  The similarity map $ \varphi $ possesses the three ingredients of   the Devaney chaos, namely density of periodic points, transitivity and sensitivity. A point $ \mathcal{F}_{i_1 i_2 i_3 ...} \in \mathcal{F} $ is periodic with period $ n $ if its index consists of endless repetitions of a block of $ n $ terms.

  The  map  $\varPhi$  admits \textit{ domain structured  Devaney  chaos} on $F,$   if the  map  $\varphi$  is Devaney  chaotic on $\mathcal F.$   	
  	
If    the   diameter and separation conditions are valid,  then  $\mathcal F$   is   said to be  a \textit{chaotic  structure}  for   $F.$  

The proof of  the  next  theorem  extends  the  technique  for  symbolic  dynamics \cite{Wiggins88}.

\begin{theorem} \label{Thm1} If $\mathcal F$ is a chaotic  structure for $F,$  then the  dynamics  of  $ \varPhi$   admits   domain structured  chaos in the sense of Devaney.
\end{theorem}

\begin{proof}
	Fix a member $ \mathcal{F}_{i_1 i_2 ... i_n ... } $ of $ \mathcal{F} $ and a positive number $ \varepsilon $. Find a natural number $ k $ such that $ \mathrm{diam}(\mathcal{F}_{i_1 i_2 ... i_k}) < \varepsilon $ and choose a $ k $-periodic element $ \mathcal{F}_{i_1 i_2 ... i_k i_1 i_2 ... i_k ...} $ of $ \mathcal{F}_{i_1 i_2 ... i_k} $. It is clear that the periodic point is an $ \varepsilon $-approximation for the considered member. The density of periodic points is thus proved.
	
	Next, utilizing the diameter condition, the transitivity will be proved if we show the existence of an element $ \mathcal{F}_{i_1 i_2 ... i_n ...} $ of $ \mathcal{F} $ such that for any subset $ \mathcal{F}_{i_1 i_2 ... i_k} $ there exists a sufficiently large integer $ p $ so that $ \varphi^p(\mathcal{F}_{i_1 i_2 ... i_n ...}) \in \mathcal{F}_{i_1 i_2 ... i_k} $. This holds  true since we can construct a sequence $ i_1 i_2 ... i_n ... $ such that it contains all the sequences of the type $ i_1 i_2 ... i_k $ as blocks.
	
	For sensitivity, fix a point $ \mathcal{F}_{i_1 i_2 ... } \in \mathcal{F} $ and an arbitrary positive number $ \varepsilon $. Due to the diameter condition, there exist an integer $ k $ and element $ \mathcal{F}_{i_1 i_2 ... i_k j_{k+1} j_{k+2} ...} \neq \mathcal{F}_{i_1 i_2 ... i_k i_{k+1} i_{k+2} ...} $ such that $ d(\mathcal{F}_{i_1 i_2 ... i_k i_{k+1} ...}, \mathcal{F}_{i_1 i_2 ... i_k j_{k+1} j_{k+2} ...}) < \varepsilon $. We choose  $ j_{k+1}, j_{k+2}, ... $ such that $ d(\mathcal{F}_{i_{k+1} i_{k+2} ... i_{k+n}}, \mathcal{F}_{j_{k+1} j_{k+2} ... j_{k+n}}) > \varepsilon_0 $, by the separation condition. This proves the sensitivity.	
\end{proof}

In \cite{AkhmetUnpredictable,AkhmetPoincare}, Poisson stable motion is utilized to distinguish  chaotic behavior from  periodic motions in Devaney and Li-Yorke types.  The dynamics  is given the named  Poincar\`{e} chaos. 

  The  map  $\Phi$  admits  on $F$ \textit{domain structured  Poincar\'{e}  chaos,}  if the  map  $\varphi$  is Poincar\'{e}  chaotic on $\mathcal F.$

 The next theorem shows that the Poincar\`{e} chaos is valid for the similarity dynamics.

\begin{theorem} \label{Thm2}  If $\mathcal F$ is a chaotic  structure for $F,$  then the  map $\varPhi$  possesses   domain structured Poincar\`{e} chaos.
\end{theorem}

The proof of the last theorem is based on the verification of Lemma 3.1 in \cite{AkhmetPoincare},  applied  to the similarity map.

In addition to the Devaney and Poincar\`{e} chaos, it can be shown that the Li-Yorke chaos is also present  in the dynamics of the map $ \varPhi $. The proof of the  theorem is similar to that of Theorem 6.35 in \cite{Chen} for the shift map defined in  the space of symbolic sequences.  In what   follows we will  say  that the dynamics  of the map  $\varPhi$  admits the   domain structured chaos.

\section{Chaotic random processes} 

  Consider a  discrete time  random    process  discrete   finite state space  as a family of random variables $\textbf X(n), n=1,2,\ldots.$   Denote  by 
$S$ the  state  space of the process,  and  consider it  with  a distance $d.$  We assume positive probability  for all  members   of the state  space,  and for  each  experiment   one of the  members of the state space must   necessarily happen.     

Suppose that  the   there exists   a chaotic structure 		     
\begin{equation} \label{DomainStrSpace}
\mathcal{S} =  \big\{\mathcal{S}_{i_1 i_2 ... i_n ... } : i_k=1,2, ..., p, \; k=1, 2, ... \big\},
\end{equation}
for  the state space.  In the  light  of the  last  section's  results,  this means that   domain-structured chaos, and,  consequently,  Poincar\'{e},  Li-Yorke and Devaney  chaos are present.   Nevertheless,  it does not  mean that  the  random  process is chaotic,   since its   dynamics may  be  richer  than that   of the  map  $\phi.$

Next, we will  consider the stochastic process,  which is chaotic,   owing  to  the   coincidence of stochastic   and deterministic dynamics.  

We assume that  the state space is finite and    denote $S= \{s_1,s_1,\ldots,s_p\}.$

Assign   for  each member $s_i, i = 1,\ldots,p,$ of the state space infinitely many  presentation elements $\mathcal {S}_{i,i_2,\ldots,i_k,\ldots}$  where  $i_j =1,\ldots,p, j=2,\ldots.$ 		      	 
Thus,  one can  consider the  pre-chaotic structure,    
$\mathcal S =  \{\mathcal {S}_{i_1,i_2,\ldots,i_k,\ldots},  i_j =1,\ldots,p, j=1,\ldots  \}.$

Consider  the  distance 		      	                                                                                                                                                                                                                                                                                                              
$d(\mathcal {S}_{i_1,i_2,\ldots,i_k,\ldots}, \mathcal {S}_{j_1,j_2,\ldots,j_k,\ldots})= d(s_{i_1},s_{j_1})$ if $i_1,  j_1 = 1,\ldots,p.$       It  is   easy  to  verify, that  the  diameter and separation conditions are  valid for  the   set  $\mathcal S.$       Since  the  set  of all  realizations coincides with    the set  of all sequences on the finite set  of numbers,     and each experiment is equivalent  to  the iteration of the shift,  the random dynamics admits  the same chaotic properties as the trajectories of the similarity map dynamics. Therefore,  the   following  theorem  is valid.
\begin{theorem} If    the  state  space is finite,  then the random  process  $X(n)$   admits the  domain-structured chaos.   
\end{theorem}

That is, random processes such as coin tossing, dice rolling, traffic lights, tetrahedron dice rolling, and five city entrance,    \cite{Doob},  are chaotic, since the collections of their realizations (sample sequences) are chaotic sets.

\section{Conclusion}   We have discussed the appearance  of deterministic chaos in the
discrete random process. We introduce   an instrument that we call a  domain structured chaos, which can have many more future applications. The  approach can be used for
non-stationary processes, continuous-time  random
processes, as well as processes with  Markov chains.
Our research suggests that the methods historically developed for  chaotic
dynamics , e.g. synchronization and control, can be extended to dynamics
with probability.

\bibliography{ref}
\bibliographystyle{plain}

\end{document}